\documentclass[12pt,reqno]{amsart}
\usepackage[T2A]{fontenc}
\usepackage[cp1251]{inputenc}
\usepackage[russian,english]{babel}

\usepackage{xcolor} 

\usepackage[pdftex,unicode,colorlinks,linkcolor=blue,
citecolor=red,bookmarksopen,pdfhighlight=/N]{hyperref}

\usepackage{indentfirst}
\usepackage{verbatim}
\usepackage{amsfonts,amssymb,mathrsfs,amscd}
\usepackage{amsmath,latexsym}
\usepackage{amsthm}
\usepackage{verbatim}
\usepackage{eucal}
\usepackage{graphicx}
\usepackage{enumerate}

\usepackage{amsbib}

\theoremstyle{plain} 
\newtheorem*{theorem}{Теорема} 

\newtheorem*{propos}{Предложение}
\newtheorem*{corollary}{Следствие}
\theoremstyle{definition}

\renewcommand{\leq}{\leqslant}

\newcommand{\RR}{\mathbb{R}}

\DeclareMathOperator{\rh}{\text{\rm \tiny rh}}
 \DeclareMathOperator{\lh}{\text{\rm \tiny lh}}

\DeclareMathOperator{\dd}{\,{\mathrm d\!}}

\renewcommand{\leq}{\leqslant}

\begin{document} 
\title{Интегральные средние\\ и максимизация функций}
	
\author[Б.\,Н.~Хабибуллин]{Б.\,Н.~Хабибуллин}
\address{Башкирский государственный университет}
\email{khabib-bulat@mail.ru}

\date{18.12.2019}
\selectlanguage{russian}
 \maketitle

\begin{abstract}  Some properties of integral averages of functions on intervals and their asymptotic behavior are investigated. The results are aimed at applications to entire and subharmonic functions.

\textbf{Bibliography:} 6 titles 


\end{abstract}

\markright{Интегральные  средние и максимизация функций}

\footnotetext[0]{Исследование выполнено за счёт гранта Российского
научного фонда (проект № 18-11-00002).}


Пусть $\RR_{\pm\infty}:=\{-\infty\}\cup \RR\cup \{+\infty\}$ --- расширение \textit{вещественной оси\/} $\RR$, $X$ --- подмножество в  $\RR_{\pm\infty}$.

Функция $f\colon X\to \RR_{\pm\infty}$ {\it возрастающая} (соответственно \textit{строго возрастающая\/}) на $X$, если для любых $x_1,x_2\in X$ из $x_1<x_2$ следует $f(x_1)\leq f(x_2)$ (соответственно $f(x_1)<f(x_2)$). 

Функция $f$ \textit{убывающая\/} (соответственно \textit{строго убывающая\/}) на $X$, если противоположная функция $-f$ {\it возрастающая} (соответственно \textit{строго возрастающая\/}) на $X$.  

Всюду в данной заметке $X:=[a,b)\subset \RR$, $a<b$. 

Пусть $m\colon [a,b)\to \RR$ --- \textit{строго возрастающая функция.\/} 
\textit{Интегральное среднее ${\mathsf A}_m(r,R; f)$ функции $f$ по функции $m$ на  интервале\/
$[r,R]\subset [a,b)$,} $r<R$,  определяем через  интеграл Стилтьеса:
\begin{equation}\label{Am}
{\mathsf A}_m(r,R;f):=\frac{1}{m(R)-m(r)}\int_r^R f\dd m, \quad a\leq r<R<b,
\end{equation}
если, конечно, интегралы Стилтьеса корректно определены. Например, достаточно, чтобы функция $f$ была \textit{интегрируема по Риману,\/} а функция $m$ была \textit{непрерывна справа в точке\/} $a$,  \textit{дифференцируема на интервале\/} $(a,b)$ с \textit{ограниченной производной} на каждом интервале $[a',b']\subset (a,b)$, что мы и будем предполагать всюду ниже.  В этом случае  при $a<r<R<b$ непосредственно вычисляются  частные производные 
\begin{subequations}\label{Amp}
\begin{align}
\frac{\partial}{\partial r} {\mathsf A}_m(r,R;f)&=
\frac{m'(r)}{\bigl(m(R)-m(r)\bigr)^2}\int_r^R\bigl(f(x)-f(r)\bigr)\dd m(x).
\tag{\ref{Amp}r}\label{{Amp}r}\\
\frac{\partial}{\partial R} {\mathsf A}_m(r,R;f)&=
\frac{m'(R)}{\bigl(m(R)-m(r)\bigr)^2}\int_r^R\bigl(f(R)-f(x)\bigr)\dd m(x). 
\tag{\ref{Amp}R}\label{{Amp}R}
\end{align}
\end{subequations}

\begin{propos}\label{prf} Пусть  $f\colon [a,b)\to \RR$ ---  убывающая функция. Тогда  в принятых  выше соглашениях интегральное среднее \eqref{Am} убывает по переменной $r$ на $[a,R)$ и по переменной   $R$ 
на $[r,b)$. При этом
\begin{equation}\label{aAm}
\sup_{r\in [a,R)} {\mathsf A}_m(r,R;f)= {\mathsf A}_m(a,R;f) \quad\text{для всех 
$R\in [a,b)$}.
\end{equation} 
\end{propos}
\begin{proof} Сразу следует соответственно из \eqref{{Amp}r} и \eqref{{Amp}R}.
\end{proof}

Для функции $f\colon [a,b)\to \RR_{\pm\infty}$  определим её 
\textit{правую\/} и {\it левую максимизации\/} 
\begin{equation}\label{mf}
{\mathsf m}^{\rh}[f](r):=\sup_{x\in [r,b)}f(x), 
\quad {\mathsf m}^{\lh}[f](r):=\sup_{x\in [a,r]}f(x), \quad r\in [a,b). 
\end{equation}
По построению, очевидно, правая максимизация ${\mathsf m}^{\rh}[f]$ --- \textit{убывающая функция,\/} левая максимизация ${\mathsf m}^{\lh}[f]$ --- \textit{возрастающая  функция,\/} и обе со свойствами
\begin{equation}\label{mflh}
f\leq \begin{cases}
{\mathsf m}^{\rh}[f],\\
{\mathsf m}^{\lh}[f]
\end{cases} \quad \text{на $[a,b)$}, 
\qquad  \sup_{[a,b)} f= 
\begin{cases}
\sup_{[a,b)}{\mathsf m}^{\rh}[f],\\
\sup_{[a,b)}{\mathsf m}^{\lh}[f].  
\end{cases}
\end{equation}

\begin{theorem}\label{thIA} Пусть  для функции $f\colon [a,b)\to \RR^+$  
\begin{equation}\label{f0m}
\sup_{[a,b)} f<+\infty, \quad \lim_{b>x\to b} f(x)=0; \quad \lim_{b>x\to b} m(x)=+\infty. 
\end{equation}
Тогда функция ${\mathsf A}_m(a,R;{\mathsf m}^{\rh}[f])$ убывающая по $R\in (a,b)$,  
\begin{subequations}\label{F}
\begin{align}
{\mathsf A}_m(r,R;f)&\leq {\mathsf A}_m(a,R;{\mathsf m}^{\rh}[f])\quad \text{при всех $a\leq r<R\in (a,b)$},
\tag{\ref{F}A}\label{F1}\\
\lim_{b>R\to b}{\mathsf A}_m&(a,R;{\mathsf m}^{\rh}[f])=0,
\tag{\ref{F}$_0$}\label{F0}
\end{align}
\end{subequations} 

В дополнение к \eqref{f0m}, пусть  возрастающая функция $n\colon [a,b)\to \RR^+$  с $n(a)>0$  ограничена на каждом интервале $[a,b']\subset [a,b)$, и
\begin{equation}\label{n}
\lim_{b>x\to b} n(x)=+\infty. 
\end{equation} 
Тогда для возрастающей функции ${\mathsf m}^{\lh}\bigl[n{\mathsf m}^{\rh}[f]\bigr]$, домноженной на $1/n$, имеем соотношения 
\begin{subequations}\label{Anm}
\begin{align}
f(R)&\leq {\mathsf A}_m\Bigl(r,R; \frac{1}{n}{\mathsf m}^{\lh}\bigl[n{\mathsf m}^{\rh}[f]\bigr]\Bigr)
\quad \text{при $a\leq r<R<+\infty$},
\tag{\ref{Anm}A}\label{{Anm}A}\\
\lim_{b>R\to b}&\frac{1}{n(R)}\Bigl({\mathsf m}^{\lh}\bigl[n{\mathsf m}^{\rh}[f]\bigr]\Bigr)(R)=0.
\tag{\ref{Anm}$_0$}\label{{Anm}0}
\end{align}
\end{subequations}
\end{theorem}
\begin{proof} Докажем первую часть. Функция ${\mathsf A}_m(a,R;{\mathsf m}^{\rh}[f])$ убывающая по $R\in (a,b)$ по Предложению. Неравенство \eqref{F1} получаем из цепочки 
\begin{multline*}
{\mathsf A}_m(r,R;f)\overset{\eqref{mflh}}{\leq} {\mathsf A}_m(r,R;{\mathsf m}^{\rh}[f])\\
\leq \sup_{r\in [a,R]} {\mathsf A}_m(r,R;{\mathsf m}^{\rh}[f])
={\mathsf A}_m(a,R;{\mathsf m}^{\rh}[f]),
\end{multline*} 
где последнее равенство --- это равенство \eqref{aAm} из Предложения.

Пусть $c>0$. Из второго  соотношения в \eqref{f0m} следует, что 
${\mathsf m}^{\rh}[f](x)=o(1)$ при $b>x\to b$, и 
существует число $R_c\in [a,b)$, для которого 
\begin{equation}\label{Rc}
{\mathsf m}^{\rh}[f](x)\leq c\quad\text{при $x\in [R_c,b)$}.
\end{equation} При $R>R_c$ по определению \eqref{Am} интегрального среднего  
\begin{multline*}
{\mathsf A}_m\bigl((a,R;{\mathsf m}^{\rh}[f]\bigr)=\frac{1}{m(R)-m(a)}\left(\int_a^{R_c}+\int_{R_c}^R \right) {\mathsf m}^{\rh}[f]\dd m\\
\leq 
\frac{\bigl(m(R_c)-m(a)\bigr)\sup_{[a,R_c]}{\mathsf m}^{\rh}[f]}{m(R)-m(a)}
+c\frac{m(R)-m(R_c)}{m(R)-m(a)}.
\end{multline*}
Отсюда при $b>R\to b$ по первому и третьему соотношениям в \eqref{f0m}  получаем
\begin{equation*}
\limsup_{b>R\to b}{\mathsf A}_m\bigl(a,R;{\mathsf m}^{\rh}[f]\bigr)\leq c.
\end{equation*}
Ввиду произвола в выборе числа $c>0$ это даёт \eqref{F0}.

Докажем вторую часть. Согласно свойствам \eqref{mf}--\eqref{mflh} функция ${\mathsf m}^{\lh}\bigl[n{\mathsf m}^{\rh}[f]\bigr]$ возрастающая как левая максимизация и 
по тем же свойствам 
\begin{multline*}
f(R)\leq {\mathsf m}^{\rh}[f] (R)
=\frac{1}{m(R)-m(r)}\int_r^R {\mathsf m}^{\rh}[f] (R) \dd m(x)\\
\leq \frac{1}{m(R)-m(r)}\int_r^R {\mathsf m}^{\rh}[f] (x) \dd m(x)
\\=\frac{1}{m(R)-m(r)}\int_r^R \frac{1}{n(x)}n(x){\mathsf m}^{\rh}[f] (x) \dd m(x)\\
\leq \frac{1}{m(R)-m(r)}\int_r^R \frac{1}{n(x)}
\Bigl({\mathsf m}^{\lh}\bigl[n{\mathsf m}^{\rh}[f]\bigr]\Bigr) (x) 
\dd m(x)\\
= {\mathsf A}_m\Bigl(r,R; \frac{1}{n}{\mathsf m}^{\lh}\bigl[n{\mathsf m}^{\rh}[f]\bigr]\Bigr),
\end{multline*}
что даёт в точности \eqref{{Anm}A}. Выберем $R_c\in [a,b)$ как в \eqref{Rc} и  с учётом  \eqref{mf}--\eqref{mflh} и с возрастающей функцией $n$, удовлетворяющей \eqref{n}, при $R>R_c$ оценим 
\begin{multline*}
\frac{1}{n(R)}\Bigl({\mathsf m}^{\lh}\bigl[n{\mathsf m}^{\rh}[f]\bigr]\Bigr)(R)
=\frac{1}{n(R)}\sup_{r\in [a,R]} n(r)\bigl({\mathsf m}^{\rh}[f]\bigr)(r)
\\
\frac{1}{n(R)}\sup_{r\in [a,R_c]}n(r)\bigl({\mathsf m}^{\rh}[f]\bigr)(r) +
\frac{1}{n(R)} \sup_{r\in [R_c,R]}n(r)\bigl({\mathsf m}^{\rh}[f]\bigr)(r)
\\
\leq
\frac{n(R_c)}{n(R)}\sup_{r\in [a,R_c]}\bigl({\mathsf m}^{\rh}[f]\bigr)(r) +
\sup_{r\in [R_c,R]}\bigl({\mathsf m}^{\rh}[f]\bigr)(r).
\end{multline*}
Отсюда при $b>R\to b$ по  \eqref{mflh} и \eqref{n}  получаем
\begin{equation*}
\limsup_{b>R\to b}\frac{1}{n(R)}\Bigl({\mathsf m}^{\lh}\bigl[n{\mathsf m}^{\rh}[f]\bigr]\Bigr)(R)\leq c.
\end{equation*}
Ввиду произвола в выборе числа $c>0$ это даёт \eqref{{Anm}0}.
\end{proof}

Следующее Следствие будет использована в ином месте для развития  уже классических результатов П.~Мальявена и Л.\,А.~Рубела \cite{MR}, \cite[22]{RC}, а также их существенных обобщений  из \cite{KhaD88}--\cite[3.2]{Khsur}.   

\begin{corollary}\label{corln} Пусть $0<r_0\in \RR^+$. 

Если  для функции $Q\colon [r_0,+\infty) \to \RR^+$ существует предел
\begin{equation}\label{diy}
\lim_{x\to +\infty}\frac{Q(x)}{x}=0,
\end{equation}
то найдётся  убывающая функция   $d\colon [r_0,+\infty)\to \RR^+$, для которой 
\begin{subequations}\label{dQ}
\begin{align}
\int_r^R\frac{Q(x)}{x^2}\dd t&\leq d(R)\ln\frac{R}{r}\quad \text{при всех $r_0\leq r<R<+\infty$},
\tag{\ref{dQ}A}\label{{ad}A}\\
\lim_{R\to +\infty} d(R)&=0. 
\tag{\ref{dQ}$_0$}\label{{ad}0}
\end{align}
\end{subequations} 
Если для функции $d\colon [r_0,+\infty) \to \RR^+$ выполнено \eqref{{ad}0}, то найдётся возрастающая функция $Q\colon [r_0,+\infty)\to \RR^+$, для которой одновременно  выполнены соотношения  \eqref{diy} и 
\begin{equation}\label{Qd}
d(R)\ln\frac{R}{r} \leq \int_r^R\frac{Q(x)}{x^2}\dd x
\quad \text{при всех $r_0\leq r<R<+\infty$}.
\end{equation}
\end{corollary}
\begin{proof} Положим $a:=r_0$, $b:=+\infty$, $m(x) :=\ln x$ ---  строго возрастающая функция на  $[a,b)$, удовлетворяющая последнему  условию в \eqref{f0m}. 

Пусть для некоторой функции $Q$ выполнено \eqref{diy}. 
Тогда  для $f(x):=Q(x)/x$, $x\in [r_0,+\infty)$, выполнены два первых условия из  
\eqref{f0m}. В обозначениях \eqref{Am}, \eqref{mf} положим 
\begin{multline*}
d(R):={\mathsf A}_m(a,R;{\mathsf m}^{\rh}[f])
\\=\frac{1}{\ln R-\ln r}\int_{r_0}^R \bigl({\mathsf m}^{\rh}[f]\bigr)(x)\dd\, \ln x, 
\quad R\in [r_0,+\infty).
\end{multline*}
Отсюда по Теореме  функция $d$ убывающая,  $$d(R)\overset{\eqref{F0}}{=}o(1), \quad R\to +\infty,$$
 а также 
\begin{multline*}
\frac{1}{\ln (R/ r)}\int_r^R\frac{Q(x)}{x^2} =
\frac{1}{\ln R-\ln r}\int_{r_0}^R f(x)\dd\, \ln x
\overset{\eqref{Am}}{=}{\mathsf A}_m(r,R;f)\\
\overset{\eqref{F1}}{\leq} {\mathsf A}_m(a,R;{\mathsf m}^{\rh}[f])
=d(R)
\end{multline*} 
при всех $a=r_0\leq r<R\in (a,b)=(r_0,+\infty)$,
что даёт в точности \eqref{{ad}A}, а также первую часть Следствия.

Пусть для некоторой функции $d\colon [r_0.+\infty)$ выполнено  \eqref{{ad}0}. 
Положим 
\begin{equation*}
n(x):=x, \; x\in [r_0,+\infty),  \quad f:=d,  \quad
Q:={\mathsf m}^{\lh}\bigl[n{\mathsf m}^{\rh}[f]\bigr].  
\end{equation*}
Тогда по второй части Теоремы  функция $Q$ возрастающая, ввиду 
\eqref{{Anm}0} выполнено \eqref{diy}, а согласно \eqref{{Anm}A}
получаем 
\begin{multline*}
d(R)=f(R)\leq {\mathsf A}_m\Bigl(r,R; \frac{1}{n}{\mathsf m}^{\lh}\bigl[n{\mathsf m}^{\rh}[f]\bigr]\Bigr)\\
\overset{\eqref{Am}}{=}\frac{1}{\ln R-\ln r}\int_r^R \frac{Q(x)}{x}\dd \ln x \\
=\frac{1}{\ln (R/r)}\int_r^R  \frac{Q(x)}{x^2} \dd x \quad \text{для всех $a=r_0\leq r<R<b=+\infty$},
\end{multline*}
что даёт в точности \eqref{Qd}.
\end{proof}

\end{document}